\documentclass[12pt]{amsart}
\usepackage[osf,sc]{mathpazo}
\usepackage{amssymb}

\usepackage{geometry}
\usepackage{bm}
\usepackage{bbm}
\usepackage{graphicx}
\usepackage{hyperref}
\hypersetup{
	colorlinks=true, 
	linktoc=all,     
	linkcolor=blue,
	citecolor=red,
	filecolor=black,
	urlcolor=blue	
}
\usepackage{enumerate}
\usepackage[inline]{enumitem}
\makeatletter
\newcommand{\inlineitem}[1][]{%
	\ifnum\enit@type=\tw@
	{\descriptionlabel{#1}}
	\hspace{\labelsep}%
	\else
	\ifnum\enit@type=\z@
	\refstepcounter{\@listctr}\fi
	\quad\@itemlabel\hspace{\labelsep}%
	\fi} \makeatother
\newcommand{\ga}{\alpha}
\newcommand{\gb}{\beta}
\newcommand{\gga}{\gamma}

\newcommand{\gth}{\theta}

\newcommand{\gl}{\lambda}

\newcommand{\gp}{\pi}

\newcommand{\gr}{\rho}

\newcommand{\gf}{\phi}

\newcommand{\Gd}{\Delta}

\newcommand{\Gl}{\Lambda}

\newcommand{\Gs}{\Sigma}

\newcommand{\Gom}{\Omega}

\newcommand{\subs}{\subset}

\newcommand{\sbnq}{\subsetneq}

\newcommand{\bs}{\backslash}

\newcommand{\ti}{\tilde}

\newcommand{\mbb}{\mathbb}

\newcommand{\us}{\underset}
\newcommand{\os}{\overset}

\newcommand{\lra}{\longrightarrow}

\newcommand{\N}{\mbb N}
\newcommand{\Z}{\mbb Z}

\newcommand{\Ra}{\Rightarrow}

\newcommand{\ora}{\overrightarrow}

\newcommand{\equ}[1]{%
	\begin{equation*}
		#1
	\end{equation*}
}
\newcommand{\equa}[1]{%
	\begin{equation*}
		\begin{aligned}
			#1
		\end{aligned}
	\end{equation*}
}
\newcommand{\equan}[2]{%
	\begin{equation}
	\label{Eq:#1}
	\begin{aligned}
	#2
	\end{aligned}
	\end{equation}
}

\theoremstyle{plain}
\newtheorem{theorem}{Theorem}[section]

\newtheorem{prop}[theorem]{Proposition}
\newtheorem{lemma}[theorem]{Lemma}
\newtheorem{cor}[theorem]{Corollary}

\newtheorem{ques}[theorem]{Question}
\newtheorem{claim}[theorem]{Claim}

\makeatletter
\def\namedlabel#1#2{\begingroup
	\def\@currentlabel{#2}%
	\label{#1}\endgroup
}
\makeatother

\newtheorem*{thmOmega}{\bf{Theorem} $\bm{\Gom}$}
\newtheorem*{thmSigma}{\bf{Theorem} $\bm{\Gs}$}
\newtheorem*{thmLambda}{\bf{Theorem} $\bm{\Gl}$}
\theoremstyle{definition}
\newtheorem{defn}[theorem]{Definition}

\theoremstyle{remark}
\newtheorem{remark}[theorem]{Remark}
\newtheorem{example}[theorem]{Example}
\numberwithin{equation}{section}
\begin{document}
\title[On Rational Sets]{On Rational Sets in Euclidean Spaces and Spheres}

\author{C P Anil Kumar}
\address{School of Mathematics, Harish-Chandra Research Institute, HBNI, Chhatnag Road, Jhunsi, Prayagraj (Allahabad), 211 019,  India. \,\, email: {\tt akcp1728@gmail.com}}



\thanks{This work is done while the author is a Post Doctoral Fellow at Harish-Chandra Research Institute, Prayagraj(Allahabad).}

\date{\sc \today}
\begin{abstract}
	 For a positive rational $l$, we define the concept of an $l$-elliptic and an $l$-hyperbolic rational set in a metric space. In this article we examine the existence of (i) dense and (ii) infinite $l$-hyperbolic and $l$-ellitpic rationals subsets of the real line and unit circle. For the case of a circle, we prove that the existence of such sets depends on the positivity of ranks of certain associated elliptic curves. We also determine the closures of such sets which are maximal in case they are not dense. In higher dimensions, we show the existence of $l$-ellitpic and $l$-hyperbolic rational infinite sets in unit spheres and Euclidean spaces for certain values of $l$ which satisfy a weaker condition regarding the existence of elements of order more than two, than the positivity of the ranks of the same associated elliptic curves. We also determine their closures. A subset $T$ of the $k$-dimensional unit sphere $S^k$ has an antipodal pair if both  $x,-x\in T$ for some $x\in S^k$. In this article, we prove that there does not exist a dense rational set $T\subs S^2$ which has an antipodal pair by assuming Bombieri-Lang Conjecture for surfaces of general type.  We actually show that the existence of such a dense rational set in $S^k$ is equivalent to the existence of a  dense $2$-hyperbolic rational set in $S^k$ which is further equivalent to the existence of a dense 1-elliptic rational set in the Euclidean space $\mbb{R}^k$. 
\end{abstract}
\maketitle
\section{\bf{Introduction}}
A metric space $(X,d)$ is rational if the distance of any two points in $X$ is rational. The subject of existence of rational subsets has an interesting history. On any line in $k$-dimensional space $\mbb{R}^k$, there exists a dense rational set. In the 1940s, S.~Ulam asked the following question.

\begin{ques}
Does there exist a dense rational subset in the plane?
\end{ques}

This was mentioned in a 1945 paper by N.~H.~Anning and P.~Erd\"{o}s~\cite{MR0013511} where they proved that if there exists an infinite integral set $X$ in the plane then all the points of $X$ must be collinear 
(a set $X$ is integral if all the members of $\Gd(X)$ are integers). By a similar  argument, we can show that we cannot have infinitely many points in $k$-dimensional space $\mbb{R}^k$ not all on a line, with all the distances being integral (refer~\cite{MR0013511}). 
The paper~\cite{MR0013511} may have been the first to exhibit a rational set in the plane whose closure is a circle. The following theorem gives a characterization of circles in the plane for which there exists a dense rational set.
\begin{theorem}
	\label{theorem:Circle}
For each positive number $\gr$, the following three conditions are equivalent.
\begin{enumerate}
\item $\gr^2$ is rational;
\item on each circle of radius $\gr$ there exists a dense rational set;
\item on each circle of radius $\gr$ there is a rational 3-set.
\end{enumerate}	 
\end{theorem}
For a proof of the above theorem, refer V.~Klee, S.~Wagon~\cite{MR1133201}, Theorem 10.2, pages 132-135. It was observed by Erd\"{o}s and a few others that subsets of the plane which contain dense and infinite rational sets must be very special. The two cases, of any general line in the plane and a circle, the square of whose radius is rational, completely determines those  irreducible algebraic curves in the plane $\mbb{R}^2$ which contain dense rational sets. In fact in 2010, J.~Solymosi, F.~Zeeuw~\cite{MR2579704} have proved the following two theorems.
\begin{theorem}
Every rational set of the plane $\mbb{R}^2$ has only finitely many points in common with an algebraic curve defined over $\mbb{R}$, unless the curve has a component which is a line or a circle, the square of whose radius is a rational. 
\end{theorem} 
\begin{theorem}
If a rational set $S$ in the plane $\mbb{R}^2$ has infinitely many points on a line (resp. a circle), then all but 4 (resp. 3) points of $S$ are on a line (resp.  a circle, the square of whose radius is rational).
\end{theorem}

Recently in 2018, J.~Shaffaf~\cite{MR3835612} proved the following theorem. 
\begin{theorem}
	\label{theorem:Shaffaf}
(Assuming the Bombieri-Lang Conjecture in arithmetic geometry for the surfaces of general type) There is no rational set that is dense in the plane.
\end{theorem}
He also proves that in this paper that if $S$ is an infinite rational set in the plane then either all but at most four points of $S$ are on a line or all but at most three points of $S$ are on a circle, the square of whose radius is rational. 

Not much work has been done on analogous higher dimensional problems of the existence of dense rational sets in $\mbb{R}^k$ for $k>2$ and $S^k$ for $k\geq 2$.
Steiger~\cite{Ste} showed that for each $k\geq 2$, $\mbb{R}^k$ contains arbitrarily large rational sets that do not lie on hyperplane. In fact as a consequence of Theorem~\ref{theorem:Circle} and Exercise 8 on page 136 of~\cite{MR1133201} whose solution is given on page 294 of~\cite{MR1133201}, we can conclude that for each rational $\gr>0$ and for each $k\geq 2$, each sphere of radius $\gr$ in $\mbb{R}^k$ contains an infinite rational set that does not lie in any hyperplane.   Here in Corollary~\ref{cor:Integral}, we prove something more for the spheres $S^k,k\geq 2$. We prove that the unit sphere $S^k$ for $k\geq 2$ contains an infinite rational set with an antipodal pair that does not lie in any hyperplane of $\mbb{R}^{k+1}$.

We now introduce a few definitions in order to state the three main Theorems~\ref{theorem:S1}, ~\ref{theorem:Equivalence},~\ref{theorem:SphereFaltings}. 
\begin{defn}
Let $(X,d)$ be a metric space. We say that $X$ is a rational set if the distance 
set \equ{\Gd(X)=\{d(x,y)\mid x,y\in X\}\subseteq \mbb{Q}^{+}\cup\{0\}.}
\end{defn}
\begin{defn}
Let $r=\frac pq, p,q \in \Z,gcd(p,q)=1$ be a rational number.
\begin{itemize}
\item We say $r$ is an elliptic rational if $q^2+p^2=\square$. 
\item We say is a hyperbolic rational if $q^2-p^2=\square$.
\end{itemize}
\end{defn}
\begin{defn}
Let $(X,d)$ be a metric space and $l>0$ be a positive rational. Let 
\equ{\mbb{Q}_{elliptic}=\{r\in \mbb{Q}\mid r\text{ is an elliptic rational}\}}
and 
\equ{\mbb{Q}_{hyperbolic}=\{r\in \mbb{Q}\mid r\text{ is a hyperbolic rational}\}.}
We say that $X$ is an $l$-elliptic rational set if the distance set 
\equ{\Gd(X)=\{d(x,y)\mid x,y\in X\}\subs \mbb{Q}^{+}\cup\{0\}}
and there exists a point $P\in X$ such that 
\equ{\Gd_P(X)=\{d(P,x)\mid x\in X\}\subseteq l(\mbb{Q}_{elliptic}^{+}\cup\{0\})=\{lr\mid r\in \mbb{Q}_{elliptic}^{+}\cup\{0\}\}.}
$P$ is called the center of ellipticity.
We say that $(X,d)$ is an 
elliptic rational set if $X$ is an $l$-elliptic rational set for some positive 
rational $l$. 

We say that $X$ is an $l$-
hyperbolic rational set if the distance set 
\equ{\Gd(X)=\{d(x,y)\mid x,y\in X\}\subs \mbb{Q}^{+}\cup\{0\}}
and there exists a point $P\in X$ such that 
\equ{\Gd_P(X)=\{d(P,x)\mid x\in X\}\subs l(\mbb{Q}_{hyperbolic}^{+}\cup\{0\})=\{lr\mid r\in\mbb{Q}_{hyperbolic}^{+}\cup\{0\} \}.}
$P$ is called the center of hyperbolicity. We say that $X$ is a 
hyperbolic rational set if $X$ is an $l$-hyperbolic rational set for some positive 
rational $l$. 
\end{defn}
\begin{defn}
Let $(X,d)$ be a metric space. We say a subset $S\subseteq X$ is a maximal $l$-elliptic rational subset if it is an $l$-elliptic rational set and there does not exist a set $X\supseteq Y\supsetneq S$ such that $Y$ is also an $l$-elliptic rational set. The definition of a maximal $l$-hyperbolic rational subset of $X$ is similar. 
\end{defn}
\begin{example}	
	\label{Example:DenseSets}
 \begin{itemize}
 \item The metric space $l \mbb{Q}_{elliptic}=\{lr\big\vert \mid r\mid\text{ is an elliptic rational}\}$ $\subset \mbb{R}$ is a dense $l$-elliptic rational set. The origin is the center of ellipticity.
\item The metric space $l \mbb{Q}_{hyperbolic}=\{lr\big\vert \mid r\mid\text{ is an hyperbolic rational}\}\subset \mbb{R}$ is a dense $l$-hyperbolic rational set. The origin is the center of hyperbolicity.
\item For a positive rational $l$, the metric space $S=\{(\frac{l(a^2-b^2)}{a^2+b^2},\frac{2lab}{a^2+b^2})\mid a,b\in \Z,a^2+b^2=\square \neq 0\}\subs S^1_l=\{(x,y)\in \mbb{R}^2\mid x^2+y^2=l^2\}$ is a $2l$-hyperbolic rational set. All points of $S$ are centers of hyperbolicity. We will show in Section~\ref{sec:DensityCircle} that this set $S$ is dense in $S^1$.
 \end{itemize}
\begin{remark}
We will give a criterion later, in Section~\ref{sec:DensityCircle} for the unit circle, as to when it does, or does not have, an infinite subset, which is an $l$-hyperbolic rational set for any $l\neq 2$, or an $l$-elliptic rational set for any $l>0$, and also determine, the closure of these infinite sets, when they are maximal subsets of $S^1$.
\end{remark}
\end{example}
\begin{defn}
Let $(X,d)$ be a metric space. We say that $X$ is a 
pythogorean rational set if $X$ is either a hyperbolic rational set or an elliptic rational set. 
\end{defn}

\subsection{The main theorem}
Now we state the main theorems of this article.

\begin{thmOmega}[First Main Theorem]
	\namedlabel{theorem:S1}{$\Gom$}
~\\
	\begin{enumerate}
		\item Let $l\neq 2$ be a positive rational. There exists an infinite $l$-hyperbolic rational subset $X$ of the unit circle $S^1$ (dense if $l>2$) if and only if the elliptic curve \equ{E^{hyperbolic}_l:y^2=x(x-1)(x-\frac{l^2}{4})} defined over rational numbers has positive rank. If $l<2$ and $E^{hyperbolic}_l$ has positive rank then  such a set $X$ which is maximal, cannot be dense in $S^1$ and its closure is an arc in $S^1$ of length $4\gth$ where, $l=2sin(\gth)$ for $\gth\in (0,\frac{\gp}{2})$.
		\item Let $l>0$ be any positive rational. There exists a dense $l$-ellitpic rational subset $X$ of the unit circle $S^1$ if and only if the elliptic curve \equ{E^{elliptic}_l:y^2=x(x+1)(x-\frac{l^2}{4})} defined over rational numbers has positive rank.
	\end{enumerate}
In both cases we can choose such sets $X$ to contain only points with rational co-ordinates and in addition $X$ is simultaneously a $2$-hyperbolic rational set with the same center.
\end{thmOmega}

\begin{thmSigma}[Second Main Theorem]
	\namedlabel{theorem:Equivalence}{$\Gs$}
	~\\
Let $k\in \N$. Let $\mbb{R}^k$ denote the k-dimensional Euclidean space and $S^k$ denote the k-dimensional unit sphere. The following are equivalent.
\begin{enumerate}
	\item There exists a dense rational subset $X\subs S^k$ containing an antipodal pair.
	\item There exists a dense 2-hyperbolic rational subset $X\subs S^k$.
	\item There exists a dense 1-elliptic rational subset $Y\subs \mbb{R}^k$.
\end{enumerate} 	
As a consequence we obtain, if it exists, a dense $2$-hyperbolic rational subset $X\subs S^k$ with an antipodal pair both of which are centers of ellipticity.
\end{thmSigma}	
In Corollaries~\ref{cor:Integral},~\ref{cor:LPythogorean}, we prove the existence of, an infinite $l$-elliptic rational set for $l\in \mbb{Q}^+$ and an $l$-hyperbolic rational set for $l\neq 2,l\in \mbb{Q}^+$, in the sphere $S^k$ for $k\geq 2$, and the subsets do not lie in any hyperplane of $\mbb{R}^{k+1}$, if and only if, the elliptic curves mentioned in first main Theorem~\ref{theorem:S1}, have an element of order greater than two, which is a weaker condition. As a consequence of second main Theorem~\ref{theorem:Equivalence}, however, we have the following main theorem, for the two dimensional sphere $S^2$, assuming Bombier-Lang conjecture, for the surfaces of general type, and using Theorem~\ref{theorem:Shaffaf}.
\begin{thmLambda}[Third Main Theorem]
	\namedlabel{theorem:SphereFaltings}{$\Gl$}
	~\\
	(Assuming the Bombieri-Lang Conjecture in arithmetic geometry for the surfaces of general type) There does not exist a dense rational subset $X\subs S^2$ containing an antipodal pair.
\end{thmLambda}
The main theorems lead to the following two questions in this subject.
\begin{ques}
\label{conj:One}
Let $1<k\in \mbb{N}$. Let 
\equ{\mbb{R}^k=\{(x_1,\ldots,x_k)\mid x_i\in \mbb{R}\}}
denote the $k$-dimensional Euclidean space. Then determine all the possible closures of maximal $l$-elliptic rational subsets $X\subs 
\mbb{R}^k$.
\end{ques}

\begin{ques}
\label{conj:Two}
Let $1<k\in \mbb{N}$. Let 
\equ{S^k=\{(x_0,x_1,\ldots,x_k)\in \mbb{R}^{k+1}\mid \us{i=0}{\os{k}{\sum}}x_i^2=1\}}
denote the $k$-dimensional unit sphere. Then determine all the possible closures of maximal $l$-hyperbolic rational subsets 
$X\subs S^k$.
\end{ques}

\begin{remark}
For the sets \equ{\mbb{R},S^1=\{(x,y)\in \mbb{R}^2\mid x^2+y^2=1\}.} 
we answer Questions~\ref{conj:One},~\ref{conj:Two} completely in this article. In the case of $\mbb{R}$ there exists dense $l$-elliptic and $l$-hyperbolic rational subsets $X$ for all $l\in \mbb{Q}^+$ as given in Example~\ref{Example:DenseSets}. In~\ref{Example:DenseSets}, the sets $X$ that are given as examples are actually subsets of rationals which are dense. It is also easy to see that any such maximal set in $\mbb{R}$ must be also dense in $\mbb{R}$. The case of the unit circle $S^1$ is considered in detail in Theorem~\ref{theorem:DensityCircleRationalset} and first main Theorem~\ref{theorem:S1} of Section~\ref{sec:DensityCircle}. Here we produce such rational subets of $S^1$ which also have rational co-ordinates.
Also refer to~\cite{MR0525760} by Paul D. Humke and Lawrence L. Krajewski for a 
characterization of circles in the plane whose rational points are dense in their respective
circles. 

The conjectures of existence of dense rational subsets of spheres of rational radii and of Euclidean spaces of dimension $k$ are of interest for $k>1$ (refer to~
\cite{MR1133201}, Chapter 1, Problem 10, Parts 1,2, especially Problem 10.8 on Page 135). Also regarding rational approximability of finite sets in the plane, a very related concept, interesting results are found in C.~P.~Anil Kumar~\cite{MR3605242}.
\end{remark}

\section{\bf{Density of rational sets on any circle with rational radius}}
\label{sec:DensityCircle}
We examine in this section the case of circles in the plane for various types of dense rational sets.
\begin{theorem}
\label{theorem:DensityCircleRationalset}
Let $C$ be any circle with rational radius $\gr>0$ in the Euclidean plane. Then $C$ has a dense $2\gr$-hyperbolic rational set.
\end{theorem}
\begin{proof}
It is enough to prove for the unit circle centered at the origin by using
arbitrary translation and rational dilation, that it has a dense 2-hyperbolic rational set.
First we prove the following claim.
\begin{claim}
\label{cl:DoublePythogorean}
~\\
Let
\equa{\mbb{Q}_{tan} = \{\gth \in \mbb{R} \mid tan(\gth)& \text{ is rational or }
\gth=(2k+1)\frac{\gp}{2}, k\in \Z\}.\\
C_{\mbb{Q}}=\{\gth\in \mbb{R}\mid cos(\gth)&,sin(\gth)\in \mbb{Q}\}.\\
\mbb{Q}_{tan2} = \{\gth \in \mbb{R} \mid tan(\gth)& = \frac{q}{p}, p,q \in\Z,
gcd(p,q) = 1,p^2+q^2 \text{ is a square or }\\
\gth&=(2k+1)\frac{\gp}{2}, k\in \Z\}.\\
\mbb{Q}_{tan4}= \{\gth \in \mbb{R} \mid tan(\gth)& = \frac{2pq}{p^2-q^2}, p,q \in\Z, 
gcd(p,q) = 1,p^2+q^2 \text{ is a square}\}.
}
Then
\begin{enumerate}
\item $\mbb{Q}_{tan2}=2\mbb{Q}_{tan}=C_{\mbb{Q}}$ and 
$\mbb{Q}_{tan4}=2\mbb{Q}_{tan2} = 4\mbb{Q}_{tan}$.
\item $\mbb{Q}_{tan},\mbb{Q}_{tan2}=C_{\mbb{Q}},\mbb{Q}_{tan4}$ are dense additive 
subgroups of $\mbb{R}$.
\end{enumerate}
\end{claim}
\begin{proof}[Proof of claim]
We prove (1) first. We use some elementary geometry. Let $C$ be a circle with center
$O$ of unit radius. Let $A,B$ be two points on the circle such that
the $arc AB$ subtends an angle $2\gth$ at the center. Extend $OA$ to
meet the circle again at $P$. Then the $\measuredangle APB = \gth$.

Now we prove $\mbb{Q}_{tan2}=C_{\mbb{Q}}$. Let $\gth \in
C_{\mbb{Q}}$; then $cos(\gth)=\frac rs, sin(\gth)= \frac uv$ for some
relatively prime integers $r,s$ and $u,v$. So we have $\frac
{r^2v^2+u^2s^2}{s^2v^2}=1$ i.e. $r^2v^2+u^2s^2=s^2v^2$. If
$Cos(\gth)=0$ then $\gth \in \mbb{Q}_{tan2}$. Let $cos(\gth) \neq
0$. Now we observe that $tan(\gth)$ is rational and if for some
$q,p$ relatively prime integers $\frac qp=tan(\gth) = \frac
{us}{rv}$. Then there exists an integer $t$ such that $us=tq$ and
$rv=tp$ so $t^2(p^2+q^2)=r^2v^2+u^2s^2=s^2v^2$. So $t^2 \mid s^2v^2
\Ra t \mid sv$ and $p^2+q^2= \big(\frac {sv}{t}\big)^2$ a perfect
square. So $\gth \in \mbb{Q}_{tan2}$. The converse is also clear;
i.e. if $\gth \in \mbb{Q}_{tan2}$ then $cos(\gth),sin(\gth)$ are
rational.

Now we prove $\mbb{Q}_{tan2}=2\mbb{Q}_{tan}$. Let $\gth \in
\mbb{Q}_{tan}$ and if $tan(\gth)$ is undefined then $\gth$ is an odd
multiple of $\frac {\gp}{2}$. So $2\gth$ is an integer multiple of
$\gp$. So $tan(2\gth)=0$ and $2\gth \in \mbb{Q}_{tan2}$. If
$tan(\gth)=0$ then $tan(2\gth)=0$ so $2\gth \in \mbb{Q}_{tan2}$. If
$tan(\gth)=\frac qp$ with $gcd(q,p)=1$ then
$tan(2\gth)=\frac{2tan(\gth)}{1-tan^2(\gth)}=\frac{2pq}{p^2-q^2}$.
We observe that $(p^2-q^2)^2+4p^2q^2=(p^2+q^2)^2$ a perfect square.
Hence if $tan(2\gth) = \frac uv$ with $gcd(u,v)=1$ then also
$u^2+v^2$ is a perfect square because there exists an integer $t$
such that $2pq=tu,p^2-q^2=tv$. So $2\gth \in \mbb{Q}_{tan2}$.
Conversely, it is also clear that if $2\gth \in \mbb{Q}_{tan2}$ and 
$\gth \neq (2k+1)\frac{\gp}{2},k\in \Z$ then $tan(\gth)=\frac{sin(2\gth)}{1+cos(2\gth)}$
is rational. i.e. $\gth \in \mbb{Q}_{tan}$.

Hence we have $\mbb{Q}_{tan2}=2\mbb{Q}_{tan}=C_{\mbb{Q}}$. 
Now we prove $\mbb{Q}_{tan4}=2\mbb{Q}_{tan2}$. If $\gth\in \mbb{Q}_{tan2},
\gth\neq k\frac {\gp}{2},k\in \Z$ and $tan(\gth)=\frac qp,p^2+q^2=\square,gcd(p,q)=1$ then 
$tan(2\gth)=\frac{2pq}{p^2-q^2}$ and $2\gth\in \mbb{Q}_{tan4}$. If $\gth=k\frac{\gp}{2}$
then $tan(2\gth)=\frac 01$ and hence again $2\gth\in \mbb{Q}_{tan4}$. Now conversely,
if $tan(2\gth)=\frac{2pq}{p^2-q^2},p^2+q^2=\square,gcd(p,q)=1$ and 
$\gth\neq k\frac{\gp}{2},k\in \Z$ then we have $tan(\gth)=\frac qp$ or $-\frac pq$ and so 
$\gth\in \mbb{Q}_{tan2}$. Hence $\mbb{Q}_{tan4}=2\mbb{Q}_{tan2}$.

We observe that $tan(0)=0,tan(-\gth)=-tan(\gth),tan(\gth+\frac{\gp}{2})=-cot(\gth)$ 
and if $\gth_1+\gth_2 \neq (2k+1)\frac{\gp}{2},\gth_i\neq (2k+1)\frac{\gp}{2},k\in \Z,i=1,2$ 
for some $k \in \Z$ then 
$tan(\gth_1+\gth_2)=\frac{tan(\gth_1)+tan(\gth_2)}{1-tan(\gth_1)tan(\gth_2)}$.
So $\mbb{Q}_{tan}$ is an additive subgroup. 

Now we prove $(2)$. First we observe that for every integer $k \in \Z$ the function
\equ{tan_k:(\gp k - \frac{\gp}{2},\gp k+ \frac{\gp}{2}) \lra
\mbb{R}, \gth \mapsto tan(\gth)} is a homeomorphism. Hence the set
$\mbb{Q}_{tan}$ is dense in $\mbb{R}$. Now we observe that any 
finite index additive subgroup of a dense additive subgroup of reals is also dense in reals.
This completes the proof of the claim.
\end{proof}
Continuing with the proof of the theorem,
we consider two points $P_1,P_2$ on the circle $S^1$ in the plane such that 
\equ{P_1 = (\frac{a^2-b^2}{a^2+b^2},\frac{2ab}{a^2+b^2}), 
P_2 = (\frac{c^2-d^2}{c^2+d^2},\frac{2cd}{c^2+d^2})} for some integers $a,b,c,d$ such that 
$a^2 +b^2, c^2 + d^2$ are nonzero squares. Then we get that 
$d(P_1,P_2)^2 = \frac{4(ad-bc)^2}{(a^2+b^2)(c^2+d^2)}$ which is a square of a 
rational and we have $(a^2+b^2)(c^2+d^2)-(ad-bc)^2=(ac+bd)^2$.
Let 
\equ{X=\{(\frac{a^2-b^2}{a^2+b^2},\frac{2ab}{a^2+b^2})\in S^1\mid a,b\in \Z,
a^2+b^2=\square \neq 0\}} 
We note that $tan(\gth)=\frac{2ab}{a^2-b^2},a,b\in \Z,a^2+b^2=\square\neq 0$ if and only if $\gth\in \mbb{Q}_{tan4}$. So using the claim we establish that $S^1$ has a dense $2$-hyperbolic rational set $X$. 
Now by arbitrary translation and rational dilation 
Theorem~\ref{theorem:DensityCircleRationalset}
follows.  Note the coordinates after translation need not 
be rational but the dense set is a rational set, also, the distances are all a fixed rational
multiple of hyperbolic rationals.

In fact in the above proof if $(\frac{a^2-b^2}{a^2+b^2},\frac{2ab}{a^2+b^2})$ is in the dense
set $X$ then so are the possible four points 
\equ{\bigg(\pm \frac{a^2-b^2}{a^2+b^2},\pm\frac{2ab}{a^2+b^2}\bigg)}
and if $ab \neq 0$ then these four points 
\equ{\bigg(\pm \frac{(\frac{a+b}{2})^2-(\frac{a-b}{2})^2}{(\frac{a+b}{2})^2+(\frac{a-b}{2})^2},\pm\frac{2(\frac{a+b}{2})(\frac{a-b}{2})}{(\frac{a+b}{2})^2+(\frac{a-b}{2})^2}\bigg)=
\bigg(\pm \frac{2ab}{a^2+b^2},\pm \frac{a^2-b^2}{a^2+b^2}\bigg)}
do not belong to the dense set.
\end{proof}
\begin{remark}
If $p\neq 0 \neq q, p,q\in \Z,gcd(p,q)=1,p^2+q^2=\square\in \Z$ and if $\frac pq=\frac cd$
for some $c,d\in \Z$ then $c^2+d^2=\square$.
If $p\neq 0 \neq q, p,q\in \Z,gcd(p,q)=1,p^2-q^2=\square\in \Z$ and if $\frac pq=\frac cd$
for some $c,d\in \Z$ then $c^2-d^2=\square$.
\end{remark}
Now we consider the case of $l$-hyperbolic and $l$-elliptic rational sets on the unit circle.
\begin{prop}
	\label{prop:ellhypcircle}
\begin{enumerate}
\item 
Let $l$ be a positive rational. Suppose $X\subs S^1$ is an $l$-hyperbolic rational set such that $(1,0)\in X\subs S^1$ is the center of hyperbolicity and $X$ has at least three elements. If $e^{i\ga}\in X$ and $\mid e^{i\ga}-1\mid=2 \mid sin(\frac{\ga}{2})\mid$ where $sin(\frac{\ga}{2})=\frac l2\frac{p}{q}$ with $p,q\in \Z,q\neq 0,gcd(p,q)=1$, where $(q^2-p^2)$ is a square of an integer  then  $(4q^2-l^2p^2)$ must be a square of a rational.
\item Let $l$ be a positive rational. Suppose $X\subs S^1$ is an $l$-elliptic rational set such that $(1,0)\in X\subs S^1$ is the center of ellipticity and $X$ has at least three elements. If $e^{i\ga}\in X$ and $\mid e^{i\ga}-1\mid=2 \mid sin(\frac{\ga}{2})\mid$ where $sin(\frac{\ga}{2})=\frac l2\frac{p}{q}$ with $p,q\in \Z,q\neq 0,gcd(p,q)=1$, where $(q^2+p^2)$ is a square of an integer  then $(4q^2-l^2p^2)$ must be a square of a rational.
\end{enumerate}  
\end{prop}
\begin{proof}
	We prove $(1)$. The proof of $(2)$ is similar.
For $j=1,2$, let $e^{i\ga_j}\in X,0\neq  \ga_1\neq \ga_2\neq 0$. Suppose $\mid e^{i\ga_j}-1\mid=2\mid sin(\frac{\ga_j}{2})\mid$ where $sin(\frac{\ga_j}{2})=\frac l2\frac {p_j}{q_j}$ with $p_j,q_j\in \Z,q_j\neq 0\neq p_j,q_j^2-p_j^2$ is a square of an integer for $j=1,2$. We have 
\equ{\mid e^{i\ga_1}-e^{i\ga_2}\mid=2\mid sin(\frac{\ga_1-\ga_2}{2})\mid\in \mbb{Q}^{+}}
Here \equa{0\neq sin(\frac{\ga_1-\ga_2}{2})&=sin(\frac{\ga_1}{2})cos(\frac{\ga_2}{2})-sin(\frac{\ga_2}{2})cos(\frac{\ga_1}{2})\\
&=\frac{lp_1}{2q_1}\frac{\sqrt{4q_2^2-l^2p_2^2}}{2q_2}-\frac{lp_2}{2q_2}\frac{\sqrt{4q_1^2-l^2p_1^2}}{2q_1}\in \mbb{Q}^{*}}
So $4q_j^2-l^2p_j^2$ must be a square of a rational for $j=1,2$. This proves (1).  	
\end{proof}
Now we prove first main Theorem~\ref{theorem:S1}.
\begin{proof}
We prove (1) first. Suppose $X\subs S^1$ is an $l$-hyperbolic rational set having at least three elements with $(1,0)\in S$ being the center of hyperbolicity. Let $(1,0)\neq e^{i\ga}\in X$ with $sin(\frac{\ga}{2})=\frac{lp}{2q}$ with $p,q\in \Z,q\neq 0\neq p,gcd(p,q)=1$. Then we have using Proposition~\ref{prop:ellhypcircle} that, not only  $q^2-p^2=\square$ but also $4q^2-l^2p^2=\square$.

Consider the case when $0\neq p$ is an even integer. Then we have $p=2mn,q=(m^2+n^2)$ for some non-zero integers $m,n\in \Z,gcd(m,n)=1$. In this case the parities of $m$ and $n$ differ, since $gcd(p,q)=1$.  Hence $4q^2-l^2p^2=\square \Ra 4(m^2+n^2)^2-l^2(2mn)^2=z^2$ for some $z\in \mbb{Q}$. So $(x,y)=(\frac mn,\frac z{2n^2})\in \mbb{Q}^2$ with $m,n$ having different parities is a solution to the equation \equan{pEvenHyperbolic}{y^2=x^4-(l^2-2)x^2+1=(x^2-lx+1)(x^2+lx+1).}

Now consider the case when $p$ is an odd integer. In this case we have $p=mn$ and $q=\frac{m^2+n^2}{2}$ for some odd integers $m,n\in \Z$. Here $\frac{p}{q}=\frac{2mn}{m^2+n^2}$. Hence $4q^2-l^2p^2=\square \Ra 4(m^2+n^2)^2-l^2(2mn)^2=z^2$ for some $z\in \mbb{Q}$. So $(x,y)=(\frac mn,\frac z{2n^2})\in \mbb{Q}^2$ with $m,n$ both having odd parity is again a solution to Equation~\ref{Eq:pEvenHyperbolic}.

Conversely, any solution $(x,y)\in \mbb{Q}^2$ to the Equation~\ref{Eq:pEvenHyperbolic} with $x=\frac mn\neq 0,\ gcd(m,n)$ $=1$ and $y=\frac{z}{2n^2}\in \mbb{Q}$ for some unique $z\in \mbb{Q}$ gives rise to a fraction $\frac pq$ where $p\neq 0\neq q$ such that $q^2-p^2=\square,\ 4q^2-l^2p^2=\square,\ gcd(p,q)=1$. Here $p=2mn$ an even integer and $q=m^2+n^2$ an odd integer if $m,n$ have different parities. Otherwise $p=mn$ an odd integer and $q=\frac{m^2+n^2}{2}$ also an odd integer if $m,n$ both have odd parities.  So if for some $\ga\in (-\gp,\gp]$, $sin(\frac{\ga}{2})=\frac{lp}{2q}$ then $cos(\frac{\ga}{2})$ is also rational.

Substitute in Equation~\ref{Eq:pEvenHyperbolic} the following transformation defined over $\mbb{Q}$ \equan{TransOne}{x=\frac{v}{u-\frac{l^2}{4}},\ y= -2u+\bigg(\frac{v}{u-\frac{l^2}{4}}\bigg)^2+1\text{ for }u\neq \frac{l^2}{4},}
with inverse transformations also defined over $\mbb{Q}$ given by
\equan{TransTwo}{u=\frac{x^2-y}{2}-\frac{l^2}{4}+1,\  v=x\bigg(\frac{x^2-y}{2}-\frac{l^2}{2}+1\bigg).}  
We get the following elliptic curve under the bijection \equ{E_l^{hyperbolic}: v^2=u(u-1)(u-\frac{l^2}{4}) \text{ with }(u,v)=(\frac{l^2}{4},0) \text{ deleted}.}
Under the transformations given by Equations~\ref{Eq:TransOne},~\ref{Eq:TransTwo}
the point $(x,y)=(0,1)$ maps to $(u,v)=(0,0)$ and the point $(x,y)=(0,-1)$ maps to $(u,v)=(1,0)$. 

If $(x=\frac mn,y=\frac z{2n^2})$ is a solution to Equation~\ref{Eq:pEvenHyperbolic}
then \equ{sin(\frac{\ga}{2})=\frac{l}{2}\frac pq=\frac l2\frac{2mn}{m^2+n^2}=\frac{lx}{x^2+1}.}

If $l<2$ then the polynomial $x^4-(l^2-2)x^2+1$ is always positive and has no real zeroes. So for all real solutions $(x,y)$ to Equation~\ref{Eq:pEvenHyperbolic} we have $sin(\frac{\ga}{2})=\frac{lx}{x^2+1}$ takes all values in the interval $[-\frac l2,\frac l2]$.
If the elliptic curve  $E_l^{hyperbolic}$ has positive rank, then the rational values $\frac{l}{2}\frac{p}{q}$ of the form $\frac{lmn}{m^2+n^2}$ where, $(\frac mn,\frac z{2n^2})$ for some $z\in \mbb{Q}$ is a solution to Equation~\ref{Eq:pEvenHyperbolic}, gives a dense subset of the interval $[-\frac l2,\frac l2]$. Now for $l=2sin(\gth),\ \gth\in (0,\frac{\gp}{2})$ we get $\ga\in (-2\gth,2\gth)$ which is of length $4\gth$.

If $l>2$ then the polynomial $x^4-(l^2-2)x^2+1$ polynomial has four real roots 
$\pm x_1,\pm x_2$ where $0<x_1=\frac 1{x_2}<1$ and $x_1=\sqrt{\frac{l^2-2-l\sqrt{l^2-4}}{2}}$. Now we observe that $x^4-(l^2-2)x^2+1$ is non-negative only on the three intervals \equ{(-\infty,-\frac 1{x_1}],[-x_1,x_1], [\frac{1}{x_1},\infty).} Since $\frac{lx}{x^2+1}=\frac{l\frac{1}{x}}{\frac{1}{x^2}+1}$ and if $(x,y)$ is a solution to Equation~\ref{Eq:pEvenHyperbolic} then $(\frac 1x,\frac y{x^2})$ is also a solution to Equation~\ref{Eq:pEvenHyperbolic}, we consider only the interval  $[-x_1,x_1]\sbnq (-1,1)$. If the elliptic curve  $E_l^{hyperbolic}$ has positive rank, then we get a dense subset of values for $sin(\frac{\ga}{2})=\frac {lx}{x^2+1}$ in  $[-\frac {lx_1}{x_1^2+1},\frac {lx_1}{x_1^2+1}]$. Since $x_1$ is a positive root of $x^4-(l^2-2)x^2+1=(x^2-lx+1)(x^2+lx+1)$ we have $\frac {lx_1}{x_1^2+1}=1$. So we get density in the whole set $(-\gp,\gp]$ for the values of $\ga$ if $l>2$.

We prove (2). Suppose $X\subs S^1$ is an $l$-elliptic rational set having at least three elements with $(1,0)\in S$ being the center of ellipticity. Let $(1,0)\neq e^{i\ga}\in X$ with $sin(\frac{\ga}{2})=\frac{lp}{2q}$ with $p,q\in \Z,q\neq 0\neq p,gcd(p,q)=1$. Then we have using Proposition~\ref{prop:ellhypcircle} that, not only  $q^2+p^2=\square$ but also $4q^2-l^2p^2=\square$.

Consider the case when $p$ is a non-zero even integer and $q$ is an odd integer. Then we have $p=2mn,\ q=(m^2-n^2)$ for some non-zero integers $m,n\in \Z,gcd(m,n)=1$. In this case the parities of $m$ and $n$ differ, since $gcd(p,q)=1$.  Hence $4q^2-l^2p^2=\square \Ra 4(m^2-n^2)^2-l^2(2mn)^2=z^2$ for some $z\in \mbb{Q}$. So $(x,y)=(\frac mn,\frac z{2n^2})\in \mbb{Q}^2$ with $m,n$ having different parities is a solution to the equation \equan{pEvenElliptic}{y^2=x^4-(l^2+2)x^2+1=(x^2-lx-1)(x^2+lx-1).}

Consider the case when $p$ is an odd integer and $q$ is a non-zero even integer. Then we have $p=mn,\ q=\frac{m^2-n^2}{2}$ for some non-zero integers $m,n\in \Z,gcd(m,n)=1$. In this case the parities of $m$ and $n$ are both odd, since $gcd(p,q)=1$. Here $\frac{p}{q}=\frac{2mn}{m^2-n^2}$. Hence $4q^2-l^2p^2=\square \Ra 4(m^2-n^2)^2-l^2(2mn)^2=z^2$ for some $z\in \mbb{Q}$. So $(x,y)=(\frac mn,\frac z{2n^2})\in \mbb{Q}^2$ with $m,n$ both having odd parities is a solution to Equation~\ref{Eq:pEvenElliptic}.

Conversely, any solution $(x,y)\in \mbb{Q}^2$ to the Equation~\ref{Eq:pEvenElliptic} with $x=\frac mn\neq 0,\ gcd(m,n)$ $=1$ and $y=\frac{z}{2n^2}\in \mbb{Q}$ for some unique $z\in \mbb{Q}$ gives rise to a fraction $\frac pq$ where $p\neq 0\neq q$ such that $q^2+p^2=\square,\ 4q^2-l^2p^2=\square,\ gcd(p,q)=1$. Here $p=2mn$ an even integer and $q=m^2-n^2$ an odd integer if $m,n$ have different parities. Otherwise $p=mn$ an odd integer and $q=\frac{m^2-n^2}{2}$ an even integer if $m,n$ both have odd parities.  So if for some $\ga\in (-\gp,\gp]$, $sin(\frac{\ga}{2})=\frac{lp}{2q}$ then $cos(\frac{\ga}{2})$ is also rational.

Substitute in Equation~\ref{Eq:pEvenElliptic} the following transformation defined over $\mbb{Q}$  \equan{TransOOne}{x=\frac{v}{u-\frac{l^2}{4}},\ y= -2u+\bigg(\frac{v}{u-\frac{l^2}{4}}\bigg)^2-1\text{ for }u\neq \frac{l^2}{4},}
with inverse transformations also defined over $\mbb{Q}$ given by
\equan{TransTTwo}{u=\frac{x^2-y}{2}-\frac{l^2}{4}-1,\  v=x\bigg(\frac{x^2-y}{2}-\frac{l^2}{2}-1\bigg).}  
We get the following elliptic curve under the bijection \equ{E_l^{elliptic}: v^2=u(u+1)(u-\frac{l^2}{4}) \text{ with }(u,v)=(\frac{l^2}{4},0) \text{ deleted}.}
Under the transformations given by Equations~\ref{Eq:TransOOne},~\ref{Eq:TransTTwo}
the point $(x,y)=(0,1)$ maps to $(u,v)=(-1,0)$ and the point $(x,y)=(0,-1)$ maps to $(u,v)=(0,0)$. 
If $E_l^{elliptic}$ has positive rank then by a similar reasoning as in $(1)$ we get a dense rational subset $X$ of $S^1$ which is $l$-ellitpic.
\end{proof}

\begin{example}
\label{Example:EllipticCurves}
	Now we give examples of elliptic curves whose ranks are known to be zero or positive in the context of first main Theorem~\ref{theorem:S1}.
	
	Consider the elliptic curve $y^2=x(x-1)(x-36)$. The point $(x,y)=(64,336)$ lies on it. This point has infinite order which can be seen as follows. If we transform the curve using $u=9x-111$ and $v=27y$ then the curve becomes $v^2=u^3-34047u-2411586$ whose discriminant is $D=-2^43^{16}5^27^2$. The point $(x,y)=(64,336)$ becomes $(u,v)=(465,9072=2^43^47)$. So we conclude that $v^2\nmid D$. Hence by Nagell-Lutz theorem $(u,v)$ and hence $(x,y)$ is not a torsion point. 
	
	Consider the elliptic curve $y^2=x(x+1)(x-256)$. The point $(x,y)=(288,1632)$ lies on it. This point has infinite order as follows. If we transform the curve using $u=x-85$ and $v=y$ then the curve becomes $v^2=u^3-21931u-1250010$ whose discriminant is $D=-2^{16}\times257^2$. The point $(x,y)=(288,1632)$ becomes $(u,v)=(203,1632=2^5\times3\times17)$. So we conclude that $v^2\nmid D$. Hence by Nagell-Lutz theorem $(u,v)$ and hence $(x,y)$ is not a torsion point. 
	
	Consider the elliptic curve $y^2=x(x-1)(x-4)$. Let us denote by $E[a,b]$ the elliptic curve group 
	$E[a,b]:=\{(x,y)\in \mbb{Q}^2\mid y^2=x^3+ax^2+bx\}\cup\{\infty\}$. In this notation the given curve is represented as $E[-5,4]$. Let $\ga: E[a,b]\lra \frac{\mbb{Q}^*}{(\mbb{Q}^*)^2}$ be the homomorphism of abelian groups given by 
	$\ga(\infty)=1,\ga(0,0)=b \mod (\mbb{Q}^*)^2,\ga(x,y)=x \mod (\mbb{Q}^*)^2$ for $x\neq 0,(x,y)\in E[a,b]$. For the curve $E[-5,4]$, we get by using D.~Husem\"{o}ller~\cite{MR0868861}, Proposition 3.1 on Page 128 in Chapter 6, Section 3,  that the image of $\ga$ has at most four elements, since the only prime divisor dividing $b=4$ is the prime $p=2$. The divisors of $4$ are $\pm 1,\pm 2,\pm 4$ with $4$ being a square. We also have that if $(x,y)\in E[-5,4]$ then $x\geq 0 \Ra -1\mod (\mbb{Q}^*)^2,-2\mod (\mbb{Q}^*)^2$ do not belong to image of $\ga$. So image$(\ga)\subseteq \{1 \mod (\mbb{Q}^*)^2,2\mod (\mbb{Q}^*)^2\}$. Hence the image of $\ga$ has at most two elements. Also the for the elliptic curve $E'[a,b]=E[-2a,a^2-4b]$ which in this case is $E'[-5,4]=E[10,9]$, the image of $\ga$ is again at most four elements, since the only prime divisor dividing $a^2-4b=9$ is the prime $p=3$. In fact here the image of $\ga$ is contained in $\{\pm 1, \pm 3\}$ since $9$ is a perfect square. So we have \equ{\mid \frac{E[a,b]}{2E[a,b]}\mid \leq \mid \text{image}(\ga) \text{ for } \ga \text{ on }E\mid \mid \text{image}(\ga) \text{ for } \ga \text{ on }E'\mid \leq 2\times 4=8} using D.~Husem\"{o}ller~\cite{MR0868861}, proof of Theorem 3.2 on Page 129 in Chapter 6, Section 3. 
	The group $E[-5,4]$ has identity element and the three elements of order 2 giving rise the group $\frac{\Z}{2\Z}\times\frac{\Z}{2\Z}=\{\infty,(0,0),(1,0),(4,0)\}$. So let $E[-5,4] = T_2\oplus T_{2'} \oplus \Z^s$ a direct sum of a 2-torsion group $T_2$ and $T_{2'}$ a complementary torsion group corresponding to odd primes and a free abelian group of rank $s$.  Then $T_2$ is not a cyclic group. So  $T_2=\frac{\Z}{2^{\gl_1}\Z}\oplus \ti{T}_2$ for some $\gl_1\in \N$ and a non-trivial 2-torsion group $\ti{T}_2$. Hence $\frac{T_2}{2T_2}$ is also not cyclic and so has at least four elements. So we get the rank is either zero or one using the bound above. This bound turns out to be not very useful. Hence we compute the cardinality of the 2-Selmer group $S^{(2)}(E[-5,4])$ using Magma with the following code in the Magma Calculator \url{http://magma.maths.usyd.edu.au/calc/}
	
\equa{
	&> E:=EllipticCurve([0,-5,0,4,0]); &> two := MultiplicationByMMap(E,2);\\
	&> mu, tor := DescentMaps(two);&> S, AtoS := SelmerGroup(two);> \#S;
}   
We obtain the following output.
\equ{4\longleftarrow \text{ cardinality of the 2-Selmer group}.}
So we conclude the cardinality of the 2-selmer group $S^{(2)}(E[-5,4])$ is four. Hence the cardinality of $\frac{E[-5,4]}{2E[-5,4]}$ is at most $4$ since it is a subgroup of the 2-Selmer group. We got a better bound.  Therefore we have \equ{4=\mid S^{(2)}(E[-5,4]) \mid \geq \big\vert\frac{E[-5,4]}{2E[-5,4]}\big\vert= \mid \frac{T_2}{2T_2}\mid*\mid \frac{\Z^s}{2\Z^s}\mid\geq 4*2^s\Ra s=0. }
So this curve has rank zero.
Now we tabulate our results as follows.

\begin{center}
	\begin{tabular}{ |c|c| } 
		\hline
		Elliptic Curve & Rank  \\ 
		\hline 
		$E^{hyperbolic}_{4}:y^2=x(x-1)(x-4),l=4$ & zero\\
		\hline
		$E_{12}^{hyperbolic}: y^2=x(x-1)(x-36),l=12$ & positive \\ 
		\hline
		$E_2^{elliptic}: y^2=x(x+1)(x-1),l=2$ & zero (as 1 is not a congruent number,\\
		& refer K.~Conrad~\cite{KConrad}). \\
		\hline
		$E^{elliptic}_{32}:y^2=x(x+1)(x-256),l=32$ & positive \\
		\hline
	\end{tabular}
\end{center}\end{example}
\section{\bf{On a homeomorphism between sphere with a point deleted and Euclidean space
of same dimension}}
\label{sec:rationality}
In this section we introduce a map between the unit sphere with a point removed and 
the Euclidean space of the same dimension which preserve rationality of distances
of a certain type. We begin with a lemma.
\begin{lemma}
\label{lemma:SphericalCosineRule} Let 
\equ{\ora{x}=(x_0,x_1,\ldots,x_k),\ora{y}=(y_0,y_1,\ldots,y_k),\ora{z}=(z_0,z_1,\ldots,z_k)} 
be three points on the unit $k$-dimensional sphere $S^k$. Let 
\equa{\ora{x}\bigodot\ora{y} &= \us{i=0}{\os{k}{\sum}}x_iy_i=Cos(c),\\ 
\ora{y}\bigodot\ora{z} &= \us{i=0}{\os{k}{\sum}}y_iz_i = Cos(a),\\ 
\ora{z}\bigodot\ora{x} &= \us{i=0}{\os{k}{\sum}}z_ix_i = Cos(b).} 
Consider the spherical triangle $xyz$ on $S^k$. Let
$spherical\measuredangle yxz = \ga,spherical\measuredangle xzy =
\gga,spherical\measuredangle xyz = \gb$. Then we have the cosine
rule of spherical trigonometry for a suitable choice of angles 
for $a,b,c,\ga,\gb,\gga$
\begin{equation}
\label{eq:SphericalCosineRule}
\begin{aligned}
Cos(a) &= Cos(b)Cos(c) + Sin(b)Sin(c)Cos(\ga)\\
Cos(b) &= Cos(c)Cos(a) + Sin(c)Sin(a)Cos(\gb)\\
Cos(c) &= Cos(a)Cos(b) + Sin(a)Sin(b)Cos(\gga)\\
\end{aligned}
\end{equation}
\end{lemma}
\begin{proof}
It is easy to see that $\ora{y}-\ora{x}Cos(c),\ora{z}-\ora{x}Cos(b)$
are the tangent vectors at $x \in S^k$ tangential to the great
circles $C_{xy}$ containing $x,y$ and $C_{xz}$ containing $x,z$
respectively. It is also easy to see that 
\equ{\Vert\ora{y}-\ora{x}Cos(c)\Vert = Sin(c),\Vert\ora{z}-\ora{x}Cos(b)\Vert = Sin(b).}
So we have
\equ{(\ora{y}-\ora{x}Cos(c))\bigodot(\ora{z}-\ora{x}Cos(b)) = Sin(b)Sin(c)Cos(\ga)}
and hence it follows that
\equ{Cos(a) = Cos(b)Cos(c)+Sin(b)Sin(c)Cos(\ga)}
Hence the other equations~(\ref{eq:SphericalCosineRule}) also follow
similarly.
\end{proof}
\begin{remark}
In the above Lemma~\ref{lemma:SphericalCosineRule} a suitable choice of angles 
for $a,b,c,\ga,\gb,\gga$ exist.
\end{remark}

Now we state a theorem which is important to prove second main Theorem~\ref{theorem:Equivalence}.
of the paper.
\begin{theorem}
\label{theorem:RationalDistancePreservation}
Let $O = (0,0,\ldots,0)=0_{k+1}\in \mbb{R}^{k+1}$ denote the origin and 
let $A=\ora{e} = (1,0,0,\ldots,0) =e^{k+1}_1\in S^k$. Let 
$P=\ora{x}=(x_0,x_1,\ldots,x_k),Q=\ora{y}=(y_0,y_1,\ldots,y_k)\in S^k\bs\{\pm\ora{e}\}$.
Let $U=\ora{u},V=\ora{v}$ denote the vectors on the unit sphere
corresponding to angular bisectors of $\measuredangle AOP$ and
$\measuredangle AOQ$ in the planes of triangles $\Gd AOP,\Gd AOQ$ respectively. Consider the 
$k$-dimensional tangent space 
\equ{\mbb{T}=\mbb{T}_e(S^k)=\{1\}\times \mbb{R}^k\subs \mbb{R}^{k+1}} 
tangent to the sphere $S^k$ at $A$. Suppose after extending the vectors 
$U=\ora{u},V=\ora{v}$, they meet the tangent space $\mbb{T}$ at the points 
$B,C$ respectively. Let $\mbb{P}$ denote the plane of the triangle $\Gd ABC$ in the 
tangent space $\mbb{T}$. If 
\equ{\measuredangle AOP,\measuredangle AOQ\in \mbb{Q}_{tan4}=4\mbb{Q}_{tan}\text{ (refer to 
Claim~\ref{cl:DoublePythogorean} for definition)}}
then the following are equivalent.
\begin{enumerate}
\item $\Vert\ora{x}-\ora{y}\Vert$ is rational.
\item $sin(\frac{\measuredangle POQ}{2})$ is rational.
\item The side length $BC$ of $\Gd ABC \subs \mbb{P}$ is rational.
\item The set $S=\{A,B,C\}\subs \mbb{P}$ is a rational set.
\end{enumerate}
As a consequence the distances from $\pm A$ to points $P,Q$ are 2-hyperbolic rationals and the distances from $A$ to points $B,C$ are 1-elliptic rationals. 
\end{theorem}
\begin{proof}
For $k=1$ the theorem is trivial. So let us assume $k>1$.
First we observe that 
\equ{\Vert\ora{x}-\ora{y}\Vert^2=(\ora{x}-\ora{y})\bigodot(\ora{x}-\ora{y})=
2(1-cos(\measuredangle POQ))=4sin^2\bigg(\frac{\measuredangle POQ}{2}\bigg).}
Hence $(1)$ and $(2)$ are equivalent.

Now we observe that if $tan(\frac{\measuredangle AOP}{2}) = \frac{s}{r},
tan(\frac{\measuredangle AOQ}{2}) = \frac{v}{u}$ such that
$r^2+s^2,u^2+v^2$ are squares then by the cosine rule of spherical
trigonometry applied to the spherical triangle $xye$ on the sphere $S^k$
we have 
\equa{
\bigg(\frac{r^2-s^2}{r^2+s^2}\bigg)\bigg(\frac{u^2-v^2}{u^2+v^2}\bigg) +
\bigg(\frac{2rs}{r^2+s^2}\bigg)\bigg(\frac{2uv}{u^2+v^2}\bigg) Cos(\measuredangle BAC) &=
\ora{x}\bigodot\ora{y}\\
\Ra sin^2\bigg(\frac{\measuredangle POQ}{2}\bigg)&=\frac{1-\ora{x}\bigodot\ora{y}}{2}\\
\Ra sin^2\bigg(\frac{\measuredangle POQ}{2}\bigg)= \frac{r^2v^2 +
s^2u^2 - 2rsuvCos(\measuredangle BAC)}{(r^2+s^2)(u^2+v^2)} &=
\frac{BC^2}{(OB^2)(OC^2)}
}
Since $r^2+s^2,u^2+v^2$ are squares the sides $OB,OC$ are rational. Hence
we conclude that $(2)$ and $(3)$ are equivalent.

The sides $AB=\frac sr,AC=\frac vu$ are already rationals. Hence $(3)$ and $(4)$ are equivalent.

Now we observe that $AP^2=\Vert\ora{x}-\ora{e}\Vert^2=4sin^2(\frac{\measuredangle AOP}{2})$
and $AQ^2=\Vert\ora{y}-\ora{e}\Vert^2=4sin^2(\frac{\measuredangle AOQ}{2})$
which are both squares of rationals because 
$\measuredangle AOP,\measuredangle AOQ\in \mbb{Q}_{tan4}$ by hypothesis. Also the distance from $-A$ to $P$ is given by $\Vert\ora{x}+\ora{e}\Vert^2=4cos^2(\frac{\measuredangle AOP}{2})$ which is also a square of a rational because $tan(\frac{\measuredangle AOP}{2})=\frac sr$ such that $s^2+r^2=\square$.
So the distances from $\pm A$ to points $P,Q$ are 2-hyperbolic rationals. The distances from $A$ to points $B,C$ are 1-elliptic rationals. This proves the theorem.
\end{proof}
Here we define a homeomorphism from the unit sphere with a point removed and the Euclidean space of the same dimension.
\begin{defn}[Definition of the homeomorphism]
\label{defn:Homeo}
Let $k>0$ be a positive integer. Let 
\equ{S^k=\{(x_0,x_1,\ldots,x_k)\mid \us{i=0}{\os{k}{\sum}}x_i^2=1\}.} 
Let $O=(0,0,\ldots,0)=0_{k+1},A=\ora{e}=(1,0,\ldots,0)=e_1^{k+1}\in \mbb{R}^{k+1}$. Let
$\mbb{R}^k$ denote the $k$-dimensional Euclidean space which is identified with 
\equ{\mbb{T}_e(S^k)=\{1\}\times \mbb{R}^k \subs \mbb{R}^{k+1}} in a standard way. 
Then consider the following homeomorphism 
\equ{\gf:S^k\bs\{-\ora{e}\}\lra \{1\}\times \mbb{R}^k}
defined as follows. Let $P=\ora{x}=(x_0,x_1,\ldots,x_k)\in S^k\bs\{-\ora{e}\}$.
Consider the unit vector $U=\ora{u}\in S^k$ which corresponds to the angular
bisector of $\measuredangle AOP$. After extending this vector suppose it meets the
tangent space $\mbb{T}_e(S^k)=\{1\}\times \mbb{R}^k$ at a point $Q$. Now define
\equ{\gf(P)=Q.} This is clearly a homeomorphism of the spaces $S^k\bs\{-\ora{e}\},\mbb{R}^k$.
We note that $\gf(A)=A$.
\end{defn}
Now we prove second main Theorem~\ref{theorem:Equivalence}.
\begin{proof}
With the notations in Definition~\ref{defn:Homeo} a 1-elliptic dense rational set 
$Y\subs \mbb{R}^k$ also containing the point $A$ such that the distance from $A$ are 1-elliptic 
rationals gets mapped to a 2-hyperbolic dense rational set $X\subs \mbb{S}^k$ containing the point $A$
such that the distances from $\pm A=\pm \ora{e}$ are $2$-hyperbolic rationals because $\gf$ is a homeomorphism.

Suppose there exists a dense rational set $X\subs \mbb{S}^k$ with antipodal pair $\{x,-x\}\subs X$. Then we can we can assume, without loss of generality, by distance preserving symmetries that $x=\ora{e}=A,-x=-\ora{e}=-A$. Then by rationality of distances we have for any $P\in X, \ sin(\frac{\measuredangle AOP}{2}),\ sin(\frac{\measuredangle AOP}{2})$ are both rationals. Hence $tan(\frac{\measuredangle AOP}{2})=\frac sr$ such that $s^2+r^2=\square$. So the set $X\bs \{-A\}$ is a dense 2-hyperbolic rational set with center of hyperbolicity $A$. 
Now using the map $\gf$ we obtain a dense 1-elliptic rational subet $Y$ in the Euclidean space $\mbb{R}^k$ with center of ellipticity as the origin. 

Conversely, if we have a dense 1-elliptic rational subset $Y$ in the Euclidean space $\mbb{R}^k$, then we can assume that the center of ellipticity is the origin and use the map $\gf$ to obtain a dense 2-hyperbolic rational subset $X$ of the sphere with center of hyperbolicity as $A$. Now we can add the point $-A$ to the dense set to obtain a dense rational subset $X\cup\{-A\}$ of the sphere containing an antipodal pair ${A,-A}$.   

This proves second main Theorem~\ref{theorem:Equivalence} .
\end{proof}
Here we prove a corollary of Theorem~\ref{theorem:RationalDistancePreservation}.
Before stating a corollary we need a definition.
\begin{defn}
Let $k>0$ be a positive integer. We say a non-empty set $X\subs \mbb{R}^k$
is an integral set if the distance set 
$\Gd(X)=\{d(x,y)\mid x,y\in \mbb{R}^k\}\subs \mbb{N}\cup\{0\}$.
\end{defn}
Now we state the corollary.
\begin{cor}
\label{cor:Integral}
Let $k>0$ be a positive integer. 
\begin{enumerate}
\item There exists an infinite 1-elliptic rational set in $\mbb{Q}^k\subs \mbb{R}^k$ not contained in any hyperplane in $\mbb{R}^k$.
\item There exists an infinite 2-hyperbolic rational set in $S^k\cap \mbb{Q}^{k+1}\subs S^k$  containing an antipodal pair with both of them as centers of hyperbolicity and the set is not contained in any 
hyperplane of $\mbb{R}^{k+1}$.
\item There exists arbitrarily large finite integral $n$-set in $\mbb{Z}^k\subs \mbb{R}^k$ 
with $n>k$ which is not contained in any hyperplane and also not contained in any $(k-1)$-
dimensional sphere in $\mbb{R}^k$.
\end{enumerate}
\end{cor}
\begin{proof}
Let $X_1\subs S^1 \cap \mbb{Q}^2\subs \mbb{R}^2$ be the dense rational set introduced in 
Theorem~\ref{theorem:DensityCircleRationalset}. Fix a rational $r_0\in \mbb{Q}^{+}$, 
such that $1+r_0^2$ is a square of a rational. 
There are a few steps involved in the proof, for lifting rational sets
to one more higher dimension.
\begin{enumerate}
\item Dilate the set $X_1\cup\{(0,0)=0_2\}$ by $r_0$ to obtain a set $Y_2=r_0X_1\cup\{(0,0)\}$ which is 1-elliptic with center of ellipticity as the origin.
\item Consider $\{1\} \times Y_2$ a rational set in $\{1\}\times \mbb{R}^2\subs \mbb{R}^3$.
\item Using the homeomorphism $\gf$ in Theorem~\ref{theorem:DensityCircleRationalset}
we obtain a 2-hyperbolic rational set $X_2\subs S^2$ with center of hyperbolicity as $e^3_1=(1,0,0)$. Now add also the point $-e^3_1=(-1,0,0)$ to $X_2$ as the distances from $(-1,0,0)$ to points of $X_2$ are also rational. The point $-e^3_1=(-1,0,0)$ is also a center of hyperbolicity for the points of $X_2$.  
\item If $Q=(1,r_0x_1,r_0x_2)\in \{1\} \times r_0X_1$, 
\equ{\ti{Q}=\frac 1{\sqrt{1+r_0^2}}(1,r_0x_1,r_0x_2)\in S^2} and if 
\equ{\ti{Q}=(cos(\ga),sin(\ga)cos(\gb),sin(\ga)sin(\gb))}
has rational coordinates then 
\equa{P=(cos(2\ga)=2cos^2(\ga)-1,sin(2\ga)cos(\gb)&=2cos(\ga)sin(\ga)cos(\gb),\\
sin(2\ga)sin(\gb)&=2cos(\ga)sin(\ga)sin(\gb))} 
also has rational coordinates and the homeomorphism
$\gf$ in Theorem~\ref{theorem:DensityCircleRationalset} takes the point $P$ to the point $Q$, 
i.e. $\gf(P)=Q$. So the points in the rational set $X_2\subs S^2\cap \mbb{Q}^3$ has rational 
coordinates since $X \subs S^1 \cap \mbb{Q}^2$ also has rational coordinates. Also $X_2$ has an antipodal pair $\pm e^3_1=(\pm 1,0,0)$.
\end{enumerate}
By induction or repeating this lifting procedure we obtain 1-elliptic rational sets \equ{Y_2,Y_3,Y_4,\ldots\text{ such that }Y_j\subs \mbb{R}^j} in higher dimensional Euclidean spaces with origins as centers of ellipticity and 2-hyperbolic rational sets \equ{X_1,X_2,X_3,\ldots\text{ such that }X_j\subs S^j,\ \pm e^{j+1}_1\in X_j} in higher dimensional unit spheres. 
We observe that the following properties remain invariant during lifting procedure.
\begin{enumerate}
\item The rational distance property between the points of the lifted sets.
\item The rationality property of the coordinates of the lifted sets.
\item By including origin $0_k$ of $\mbb{R}^{k}$ in the rational set $Y_k$ for $k\geq 2$ and the points $\pm e^{k+1}_1\in \mbb{R}^{k+1}$ in the rational set $X_k$ for $k\geq 2$, the points of $Y_k$ do not lie on any hyperplane of $\mbb{R}^{k}$ for $k\geq 2$
and the points of $X_k$ do not lie on any hyperplane in $\mbb{R}^{k+1}$ for $k\geq 2$. 
\item Because the coordinates are all rational in the infinite rational sets $X_k,k\geq 1$ and $Y_k,k\geq 2$
we can obtain arbitrarily large finite sets with integer coordinates with non-negative integer 
distances by clearing denominators.
\end{enumerate}
This proves the corollary.
\end{proof}
Now we prove another corollary of Theorem~\ref{theorem:RationalDistancePreservation}.

\begin{cor}
\label{cor:LPythogorean}
Let $l$ be a positive rational.
\begin{enumerate}
\item There exists an  
an infinite $l$-elliptic rational set $X$ in $S^k\cap \mbb{Q}^{k+1}\subs \mbb{S}^k$ and it is not contained in any hyperplane of $\mbb{R}^{k+1}$ if and only if the elliptic curve \equ{E^{elliptic}_l: y^2=x(x+1)(x-\frac{l^2}{4})} has an element of order greater than two over rationals.  In particular if $E^{elliptic}_l$ has positive rank then such a set $X$ exists.
\item	Suppose $l\neq 2$. There exists an  
an infinite $l$-hyperbolic rational set $X$ in $S^k\cap \mbb{Q}^{k+1}\subs \mbb{S}^k$ and it is not contained in any hyperplane of $\mbb{R}^{k+1}$ if and only if the elliptic curve \equ{E^{hyperbolic}_l: y^2=x(x-1)(x-\frac{l^2}{4})} has an element of order greater than two over rationals.  In particular if $E^{hyperbolic}_l$ has positive rank then such a set $X$ exists.
\end{enumerate}	
In both cases the we can find a set $X$ which is also simultaneously 2-hyperbolic and the closure of $X$ is homeomorphic to a disjoint union of a circle and discrete set.
\end{cor}
\begin{proof}
The proof is similar to the proof of Corollary~\ref{cor:LPythogorean}, except that we do not add the antipode of $e^j_1$ to the set $X_j$ and the choice of $r_0$ is more specific. In the proof of Corollary~\ref{cor:LPythogorean}, if $cos(\ga)=\frac{1}{\sqrt{1+r_0^2}}$ and hence $r_0=\mid tan(\ga)\mid$ then the distance from $e^{j+1}_1$ to any point of $X_j\subs S^j\subs \mbb{R}^{j+1}$ is $2sin(\frac{\ga}{2})$. So if we choose $r_0$ or in other words if we choose $\ga$ such that  $sin(\frac{\ga}{2})=\frac{l}{2}\frac pq\neq 0$ where $q^2-p^2=\square,4q^2-l^2p^2=\square$ for $X_j$ to be $l$-hyperbolic and  $q^2+p^2=\square,4q^2-l^2p^2=\square$ for $X_j$ to be $l$-elliptic then we are done.  Such a choice of $r_0$ is guaranteed if and only if the corresponding elliptic curve has a point of order greater than two or of infinite order. 
\end{proof}
\begin{example}
There do not exist non-zero rationals such that $y^2+x^2=\square,y^2-x^2=\square$.
There are no points on the curve $E^{elliptic}_2:y^2=x(x+1)(x-1)$ of order greater than two since $1$ is not a congruent number (refer K.~Conrad~\cite{KConrad}).

Similarly we have in Example~\ref{Example:EllipticCurves},
the curve $E^{hyperbolic}_4:y^2=x(x-1)(x-4)$ has no points of order greater than 2. As we have already seen, it has no points of infinite order. For $p=5$ the curve reduces to $y^2=x(x-1)(x+1)$ which is non-singular and has the following points
\equ{\{\infty,(0,0),(\pm 1,0),(2,\pm 1),(3,\pm 2)\}\text{ over }\mbb{F}_5 } giving rise to a group of order $8$. Now using D.~Husem\"{o}ller~\cite{MR0868861}, Theorem 5.1(2), on Page 116, Chapter 5, Section 5 due to Nagell-Lutz, the map $r_5: (E^{hyperbolic}_4)_{Tor} \lra (E^{hyperbolic}_4)_{p=5}$ is injective. Hence there are no $q$-torsion points over rationals for a prime $q\neq 2$. There is no element of order 8 because the torsion group over the rationals whose cardinality divides 8 is not cyclic, since $\frac{\Z}{2\Z}\times \frac{\Z}{2\Z}$ is a subgroup of the torsion group over rationals.
Now we need to show that there are no torsion points of order four. An order two point has the y-co-ordinate zero. The only possible way that a line passing through an order two point is tangent at another point of the curve $y^2=x(x-1)(x-4)$ is that the order two point must not lie on the oval part of the curve. Hence the order two point must be $(4,0)$. So if we solve for the x-co-ordinate of the point of tangency of a line passing through $(4,0)$ we get $x=(4\pm 2\sqrt{3})$ which are both irrational. Hence there are no points of order four.

\end{example}
\section{\bf{An open question}}
\label{sec:Question}
We have observed the cases of lines and circles in the plane in this article for the existence of infinite and dense $l$-elliptic, $l$-hyperbolic rationals sets. We have also observed infinite $l$-elliptic, $l$-hyperbolic rational sets in higher dimensional unit spheres and Euclidean spaces.

Now we pose an interesting question in this section.
\begin{ques}
\label{question:typerationalapprox}
Let $k$ be a positive integer,\ $l$ be a positive rational. Let $f\in \mbb{Z}[X_1,X_2,\ldots,X_k]$ be a polynomial
with integer coefficients in $k$ variables. Let 
\equ{X=\{(x_1,x_2,\ldots,x_k)\in \mbb{R}^k\mid f(x_1,x_2,\ldots,x_k)=0\}}
be an infinite set. Let $(X,d) \subs (\mbb{R}^k,d)$ be the metric space, where $d$ is the usual
Euclidean metric. Determine whether there exists an infinite $l$-elliptic, $l$-hyperbolic rational set in $(X,d)$. Determine the closure of such a set when the set is maximal.
\end{ques}

\end{document}